\documentclass{amsart}
\usepackage{amsmath}%
\usepackage{amsthm}%

\usepackage{amsfonts}%
\usepackage{amssymb}%
\usepackage{graphicx}
\newtheorem{theorem}{Theorem}[section]
\theoremstyle{plain}

\newtheorem{corollary}[theorem]{Corollary}

\newtheorem{example}{Example}

\newtheorem{lemma}[theorem]{Lemma}

\newtheorem{problem}{Problem}[section]

\numberwithin{equation}{section}

\newcommand{\Zfrak}{\mathfrak{Z}}
\newcommand{\Mfrak}{\mathfrak{M}}
\newcommand{\Nfrak}{\mathfrak{N}}
\newcommand{\Dfrak}{\mathfrak{D}}

\newcommand{\Acal}{\mathcal{A}}
\newcommand{\Hcal}{\mathcal{H}}
\newcommand{\Mcal}{\mathcal{M}}
\newcommand{\Ncal}{\mathcal{N}}
\newcommand{\Ucal}{\mathcal{U}}

\newcommand{\Z}{\mathbb{Z}}
\newcommand{\C}{\mathbb{C}}
\newcommand{\F}{\mathbb{F}}
\newcommand{\Q}{\mathbb{Q}}
\newcommand{\R}{\mathbb{R}}
\newcommand{\N}{\mathbb{N}}
\newcommand{\A}{\mathbb{A}}
\newcommand{\G}{\mathbb{G}}
\newcommand{\cha}{\mathrm{char}\,}
\newcommand{\Aut}{\mathrm{Aut}}

\begin{document}
\title[Entire functions of positive characteristic]{Uniform positive existential interpretation of the integers in rings of entire functions of positive characteristic}
\author{Natalia Garcia-Fritz}
\address{Department of Mathematics and Statistics, Queen's University \newline%
\indent Jeffery Hall, University ave.\newline
\indent Kingston, ON Canada, K7L 3N6}%
\email[N. Garcia-Fritz]{ngarciafritz@gmail.com}%
\author{Hector Pasten}
\address{Department of Mathematics, Harvard University \newline%
\indent 1 Oxford Street,\newline
\indent Cambridge, MA 02138 USA}%
\email[H. Pasten]{hpasten@gmail.com}%
\thanks{The first author was partially supported by a Becas Chile Scholarship}
\thanks{The second author was partially supported by an Ontario Graduate Scholarship and by a Benjamin Peirce Fellowship.}
\date{\today}

\begin{abstract}
We prove a negative solution to the analogue of Hilbert's tenth problem for rings of one variable non-Archimedean entire functions in any characteristic. In the positive characteristic case we prove more: the ring of rational integers is uniformly positive existentially interpretable in the class of $\{0,1,t,+,\cdot,=\}$-structures consisting of positive characteristic rings of entire functions on the variable $t$. From this we deduce uniform undecidability results for the positive existential theory of such structures. As a key intermediate step, we prove a rationality result for the solutions of certain Pell equation (which a priori could be transcendental entire functions). 
\end{abstract}

\maketitle


\section{Introduction and results}\label{SecIntro}

In this work we establish a negative solution to the analogue of Hilbert's tenth problem for rings of one variable non-Archimedean entire functions in any characteristic (such a result for one variable non-Archimedean entire functions was only known in characteristic zero, see \cite{LipPhe}). Moreover, we prove that the ring $\Z$ is \emph{uniformly} positive existentially interpretable in the class of rings of one variable entire functions in positive characteristic, which leads to uniform undecidability results in the positive characteristic case (extending some results from \cite{PPV}). Let us introduce some notation before stating our results in a precise way.

Let $R$ be an integral domain endowed with an absolute value $|\cdot|$. We denote by $\Acal_R$ entire functions defined over $R$, by which we mean the ring of power series in the variable $t$ with coefficients in $R$ and infinite radius of convergence. More precisely, when $|\cdot|$ is Archimedean this ring is defined as
$$
\Acal_R=\left\{\sum_{n\ge 0}c_n t^n : \forall r\in\R^+, \lim_{M\to\infty}\sum_{n\ge M}|c_n|r^n=0\right\} 
$$
while when $|\cdot|$ is non-Archimedean this ring is defined as
$$
\Acal_R=\left\{\sum_{n\ge 0}c_n t^n : \forall r\in\R^+, \lim_{M\to \infty} |c_M|r^M=0\right\}.
$$
We remark that the elements of $\Acal_R$ define (by evaluation) functions $R\to R$ when $R$ is complete, but not in general. Also, there is a slight abuse of notation (the absolute value is not explicit in $\Acal_R$) but this should not lead to any confusion in the present work.

For different choices of $R$ one recovers some familiar examples of $\Acal_R$. For instance, $\Acal_{\C}$ is the ring $\Hcal$ of holomorphic functions on $\C$, while $\Acal_{\C_p}$ is the ring of analytic functions $\Hcal_p$ on $\C_p$ (as usual, $\C_p$ denotes the completion of the algebraic closure of $\Q_p$). Also, observe that any ring $R$ can be given the trivial absolute value, in which case $\Acal_R=R[t]$. The next example shows that when the absolute value is non-trivial, $\Acal_R$ can contain $R[t]$ strictly, even in positive characteristic (which is the most relevant case for our purposes).

\begin{example} Let $p$ be a prime, and let $R=\mathbb{F}_p[[q]]$ where $q$ is a variable. We denote by $v:R\to\mathbb{Z}\cup\{\infty\}$ the valuation defined by the order at $q$, and consider the absolute value defined for $f\in R$ as $|f|=p^{-v(f)}$. The power series
$$
\Theta(t)=\sum_{n=1}^\infty q^{n^2}t^n\in R[[t]]
$$
is an element of $\Acal_R$ and it is not in $R[t]$. Incidentally, note that $\Theta(t)$ is essentially the same as the (formal) reduction modulo $p$ Jacobi's theta function
$$
\vartheta(q,t):=\sum_{n\in\mathbb{Z}} q^{n^2}t^n = \Theta(t)+1+\Theta(1/t).
$$
\end{example}

The arithmetic of $\Acal_R$ shares many similarities with the arithmetic of $\Z$. This analogy is classically known in the case of $\Acal_{\F_q}=\F_q[t]$. For $R=\C$, and for non-Archimedean (complete, algebraically closed) fields $R$ in any characteristic with non-trivial absolute value, this analogy is the object of much of current research in the context of Nevanlinna theory and Vojta's analogy in both the complex \cite{VojtaCIME} and the non-Archimedean cases \cite{Wangetal, HuYang}. Therefore, in view of the negative answer to Hilbert's tenth problem on $\Z$, it is natural to ask whether the solvability of Diophantine equations over $\Acal_R$ is decidable or not. More precisely, if we consider the language $L_t=\{0,1,t,+,\cdot,=\}$ and regard $\Acal_R$ as an $L_t$-structure in the obvious way, then one can formulate the following problem.

\begin{problem}\label{ProblemH10Ana}
Let $R$ be an integral domain endowed with an absolute value $|\cdot |$. Is the positive existential theory of $\Acal_R$ over $L_t$ decidable?
\end{problem}

See \cite{PhZaSurvey} for a reference on Hilbert's tenth problem over other rings of functions (specially Sections 5-9 for rings of analytic functions), and see Section \ref{SecOpen} below for some related open problems. We remark that working over the language $L_t$ corresponds to considering Diophantine equations with coefficients in $R'[t]\subseteq \Acal_R$ where $R'$ is the image of $\Z$ in $R\subseteq \Acal_R$; this is a natural setting for Hilbert's tenth problem over rings of functions on the variable $t$ because $R'[t]$ is not `too large' meaning that it is a recursive ring (in particular, the decidability problem for equations with coefficients in $R'[t]$ makes sense) and it is not `too small' meaning that it contains elements of $\Acal_R$ that are transcendental over $R$. This last point is relevant because polynomial equations with coefficients in $R'$ (or even in $R$) have a solution in $\Acal_R$ if and only if they have a solution in $R$ (indeed, any solution in $R$ is also in $\Acal_R$, and any solution in $\Acal_R$ gives a solution in $R$ by setting $t=0$), which in many natural cases leads to a decidable Diophantine problem (for instance, if $R$ is an algebraically closed field).  We also mention that the decidability problem of existence of solutions for for \emph{single} Diophantine equations with coefficients in $R'[t]$ or for \emph{systems} of Diophantine equations with coefficients in $R'[t]$ (which is equivalent to Problem \ref{ProblemH10Ana}) are indeed equivalent, because one can verify that the fraction field of $\Acal_R$ is not algebraically closed (see Lemma 1.6 in \cite{PhZaSurvey}).

Let us briefly recall what is known for Problem \ref{ProblemH10Ana}.

One can consider two cases: the absolute value is Archimedean, or it is non-Archimedean. In the first case the problem remains wide open (despite the efforts in \cite{PheidasHolo} which contains an unfortunate mistake, see \cite{PhZaSurvey}). In the second case, Denef has proved \cite{Denef0,Denefp} that the Diophantine problem of polynomial rings in one variable over an integral domain on the language $L_t$ is undecidable; this covers the case when $|\cdot |$ is trivial. Lipshitz and Pheidas \cite{LipPhe} proved that the Diophantine problem of $\Hcal_p$ on the language $L_t$ is undecidable, while no result in positive characteristic has been established when $|\cdot |$ is non-trivial. 

Our first result covers all cases when the absolute value is non-Archimedean.

\begin{theorem}\label{MainThm}
Let $R$ be an integral domain endowed with a non-Archimedean absolute value $|\cdot |$. The positive existential theory of $\Acal_R$ in the language $L_t$ is undecidable.
\end{theorem}

This theorem is proved in Section \ref{SecFixedChar}. 

It is not incorrect to say that the same proof of \cite{LipPhe} works whenever $\cha R=0$ and $|\cdot |$ is non-Archimedean, not just for $\C_p$ (see also Theorem \ref{ThmChar0}). Therefore, our main contribution is the case of positive characteristic. Actually, the next theorem (which is our main theorem) gives a much more precise result for positive characteristic (see Section \ref{SecPPV} for a brief review of uniform interpretability).

\begin{theorem}\label{MainUnif}
The ring of rational integers $(\Z;0,1,+,\cdot,=)$ is uniformly positive existentially interpretable in the class of $L_t$-structures consisting of all $\Acal_R$, as $R$ ranges over integral domains of positive characteristic with a non-Archimedean absolute value. 
\end{theorem}

This result is proved in Section \ref{SecFinal}.

Theorem \ref{MainUnif} extends part of the main results of \cite{PPV} (namely, Theorem 1.1, Item 2) from the polynomial case to the analytic case of entire functions, because the latter contains the polynomial case by using trivial absolute values. 

As in \cite{PPV}, one can deduce strong undecidability results from uniform interpretability. Since the positive existential theory of the ring $\Z$ is undecidable (by Matijasevic, after Davis, Putnam and Robinson \cite{Matijasevich}), the following uniform undecidability result follows (see Corollary \ref{GeneralUnifUndec}).

\begin{theorem} The following problems are undecidable:

Let $C$ be a given non-empty collection of rings $R$ of positive characteristic with a non-Archimedean absolute value. Given a closed positive existential $L_t$-formula $\phi$, decide whether or not $\Acal_R$ satisfies $\phi$ for
\begin{itemize}
\item each $R\in C$;
\item at least one $R\in C$;
\item all but finitely many $R\in C$ (assuming that $C$ is infinite);
\item infinitely many $R\in C$ (assuming that $C$ is infinite).
\end{itemize}
\end{theorem}

In particular, when $C$ is a singleton we recover Theorem \ref{MainThm} in the positive characteristic case, and when $C$ only contains positive characteristic rings with the trivial absolute value we recover Item 2 of Theorem 1.2 \cite{PPV}. 

Although it is equivalent, it might be convenient to reformulate the previous theorem in terms of systems of equations instead of $L_t$-formulas.

\begin{theorem} The following problems are undecidable:

Let $C$ be a given non-empty collection of rings $R$ of positive characteristic with a non-Archimedean absolute value. Given a system $S$ of polynomial equations with coefficients in $\Z[t]$, decide whether or not $S$ has solutions in $\Acal_R$ (after reducing the coefficients modulo $\cha R$) for
\begin{itemize}
\item each $R\in C$;
\item at least one $R\in C$;
\item all but finitely many $R\in C$ (assuming that $C$ is infinite);
\item infinitely many $R\in C$ (assuming that $C$ is infinite).
\end{itemize}
\end{theorem}

As pointed out in \cite{PPV}, results of this type are in strong contrast with results such as Ax's theorem \cite{Ax} about the existence of an algorithm that decides the solvability modulo every prime of systems of Diophantine equations with coefficients in $\mathbb{Z}$.
 

Let us briefly describe the plan of the proof. As a starting point, we prove that the entire solutions of a certain Pell equation $X^2-(t^2-1)Y^2=1$ are actually polynomials (see Theorem \ref{ThmPell} -- this result can be of independent interest, and the analogue for complex entire functions is false). Note that the characteristic zero results of \cite{LipPhe} are also based on this fact (proved in \cite{LipPhe} for characteristic zero), however, our analysis of the Pell equation is substantially different and works in positive characteristic. Unlike the characteristic zero case \cite{LipPhe}, the previously mentioned result on the equation $X^2-(t^2-1)Y^2=1$ is not enough for invoking Denef's proofs from \cite{Denef0,Denefp} since the positive characteristic case would require us to analyze a more general Pell equation which seems to be beyond our methods, so, this idea does not work directly. Instead, we use our result on Pell equations to define a suitable set, which allows us to apply the main result of Pheidas \cite{PheidasII} (see Theorem \ref{ThmCharp}). Unfortunately, Pheidas' result from \cite{PheidasII} is not uniform in the characteristic and more work is needed for proving Theorem \ref{MainUnif}. In order to obtain uniformity, we go back to the idea of applying Denef's proof from \cite{Denefp} and replace the use of more general Pell equations by an application of B\"uchi's problem for polynomials in positive characteristic (see Section \ref{SecBuchi} for the latter). However, this approach only works for characteristic $p\ge 17$ (Theorem \ref{ThmAlmost}). Actually this is not a serious problem since we can cover the finitely many remaining characteristics using Theorem \ref{ThmCharp} (the application of Pheidas' results \cite{PheidasII} discussed above).


\section{Uniform interpretation}\label{SecPPV}

In this section we briefly recall the notion of uniform interpretation. For a more detailed discussion on the material in this section, the reader can consult \cite{PPV}.

Given a language $L$ and an $L$-structure $\Mfrak$ we will denote the base set of $\Mfrak$ by $|\Mfrak|$. 

The language $L$ has three types of symbols; constant, relation and function symbols. We recall the notion of \emph{realization} of a formula. If $c\in L$ is a constant symbol, then we denote its realization by $c^{\Mfrak}$ which is a singleton subset of $|\Mfrak|$ whose only element is the interpretation of $c$. If $R\in L$ is an $n$-ary relation symbol, its realization $R^{\Mfrak}$ is the subset of $|\Mfrak|^n$ given by the interpretation of $R$.  If $f\in L$ is an $n$-ary function symbol, its realization $f^\Mfrak$ is the subset of $|\Mfrak|^{n+1}$ given by the graph of the interpretation of $f$. In general, the realization of an $L$-formula $\phi=\phi(x_1,\ldots,x_n)$ is
$$
\phi^\Mfrak=\{(m_1,\ldots, m_n)\in |\Mfrak|^n : \Mfrak \mbox{ satisfies }\phi(m_1,\ldots,m_n)\}.
$$

Let $L, L'$ be languages, $\Mfrak$ an $L$-structure and $\Ncal$ a class of $L'$-structures (which we always assume non-empty). We say that $\Mfrak$ is \emph{uniformly interpretable} (without parameters) in the class $\Ncal$ if there is an $L'$-formula $\phi_L$, and for each symbol $s\in L$ there is an $L'$-formula $\phi_s$ such that, for each $\Nfrak\in\Ncal$ there is a surjective map
$$
\theta_\Nfrak :\phi_L^{\Nfrak}\to |\Mfrak|
$$
which satisfies the following conditions:
\begin{itemize}
\item $\phi_c^\Nfrak\subseteq \phi_L^{\Nfrak}$ and $\theta_\Nfrak^{-1}(c^{\Mfrak})=\phi_c^{\Nfrak}$ for each constant symbol $c$,
\item $\phi_R^\Nfrak\subseteq (\phi_L^{\Nfrak})^n$ and $(\theta_\Nfrak^n)^{-1}(R^{\Mfrak})=\phi_R^{\Nfrak}$ for each $n$-ary relation symbol $R$, and
\item $\phi_f^\Nfrak\subseteq (\phi_L^{\Nfrak})^{n+1}$ and $(\theta_\Nfrak^{n+1})^{-1}(f^{\Mfrak})=\phi_f^{\Nfrak}$ for each symbol of $n$-ary function.
\end{itemize}
One can explicitly refer to these formulas by saying that $\Mfrak$ is uniformly interpretable in $\Ncal$ \emph{by the set of formulas $\Phi=\{\phi_s\}_{s\in L\cup\{L\}}$}. We say that this interpretation is \emph{positive existential} if each formula in $\Phi$ is positive existential.

One of the main applications of uniform interpretations is related to \emph{uniform undecidability} results (here we need to assume that the languages are encoded in the integers; this is not a problem since in our applications we only consider finite languages).

\begin{theorem}
Suppose that the $L$-structure $\Mfrak$ is uniformly interpretable in a class of $L'$-structures $\Ncal$. There is a Turing machine $T$ that takes as input closed $L$-formulas $\phi$ and gives as output closed $L'$-formulas $T(\phi)$ such that the following three items are equivalent:
\begin{itemize}
\item $\Mfrak$ satisfies $\phi$, 
\item there is some $\Nfrak$ in $\Ncal$ such that $\Nfrak$ satisfies $T(\phi)$,
\item for all $\Nfrak$ in $\Ncal$ we have that $\Nfrak$ satisfies $T(\phi)$.
\end{itemize}
Moreover, if the uniform interpretation is positive existential, then $T$ takes positive existential $L$-formulas to positive existential $L'$-formulas.
\end{theorem}

\begin{corollary}\label{GeneralUnifUndec} Suppose that the $L$-structure $\Mfrak$ is uniformly (positive existentially) interpretable in a class of $L'$-structures $\Ncal$. Assume that the (positive existential) theory of $\Mfrak$ over $L$ is undecidable.  Then the following four problems are undecidable:

Let $\Ncal'$ be a non-empty subclass of $\Ncal$. Given a closed (positive existential) $L'$-formula $\phi$, decide whether or not $\Nfrak$ satisfies $\phi$ for
\begin{itemize}
\item at least one $\Nfrak$ in $\Ncal'$;
\item each $\Nfrak$ in $\Ncal'$;
\item all but finitely many $\Nfrak$ in $\Ncal'$ (this requires $\Ncal'$ infinite);
\item infinitely many $\Nfrak$ in $\Ncal'$ (this requires $\Ncal'$ infinite).
\end{itemize}
\end{corollary}

The four problems in the previous corollary do not exhaust all possible undecidability consequences, but they seem to be the most natural ones.


\section{Non-Archimedean analytic functions}
As a general reference on non-Archimedean analytic functions including the case of positive characteristic, the reader might find useful \cite{HuYang}. However, we prefer to give mostly a self-contained presentation of preliminary material for our applications in this section.

\subsection{One variable analytic functions and Newton polygons}

Let $R$ be an integral domain, complete with respect to a non-Archimedean absolute value $|\cdot |$. Given real numbers $0\le a<b$ define $\Acal_R[a,b]$ as the ring of formal series in the variable $t$ 
$$
\sum_{n\in\Z} c_nt^n
$$
with coefficients in $R$ and such that for every $r\in[a,b]$ one has 
$$
\lim_{|n|\to \infty}|c_n|r^n=0.
$$
Here, of course, we follow the convention that if $a=0$ then $c_n=0$ for $n<0$. The completeness of $R$ is needed for the multiplication to be defined on $\Acal[a,b]$, but this requirement is no longer necessary when $a=0$. The next lemma is well-known but we were not able to find a reference.
\begin{lemma}
For every $r\in [a,b]$ with $r>0$, the absolute value on $R$ extends to a non-Archimedean absolute value $|\cdot |_r$ on $\Acal_R[a,b]$ as follows: if $h=\sum_nc_nt^n$ then
$$
|h|_r= \max_n |c_n|r^n.
$$
In particular, $\Acal_R[a,b]$ is an integral domain.
\end{lemma}
\begin{proof}
That $|h|_r=0$ if and only if $h=0$ is clear. Let $g=\sum_n a_nt^n$. For the strong triangle inequality we have
$$
|h+g|_r=\max_n|c_n+a_n|r^n\le\max_n\max\{|c_n|,|a_n|\}r^n=  \max\{\max_n|c_n|r^n,\max_n|a_n|r^n\}.
$$
The function $r\mapsto |h|_r$ is continuous, hence, it suffices to prove $|fg|_r=|f|_r|g|_r$ for $r$ away from a numerable set. So we can assume that there are unique $i$ and $j$ (depending on $r$) such that $|h|_r=|c_i|r^i$ and $|g|_r=|a_j|r^j$. Note that $hg=\sum_n b_nt^n$ where
$$
b_n=\sum_{u+v=n}c_ua_v
$$
(this can be a convergent infinite series when $a>0$). The strong triangle inequality gives for each $n$
$$
|b_n|r^n\le r^n\max_u|c_ua_{n-u}|=\max_u |c_u|r^u\cdot |a_{n-u}|r^{n-u}\le |c_i|r^i\cdot |a_j|r^j=|h|_r|g|_r
$$
hence $|hg|_r\le |h|_r|g|_r$. However, among the pairs $(u,v)$ with $u+v=i+j$ we have that $(u,v)=(i,j)$ gives a strict maximum for the quantity $|c_ua_v|$ thanks to our assumptions on $i$ and $j$. The strong triangle inequality then gives
$$
|b_{i+j}|r^{i+j}= \left|\sum_{u+v=i+j}c_ua_{v}\right|r^{i+j}=|c_ia_j|r^{i+j}=|h|_r|g|_r
$$
and hence $|hg|_r\ge |h|_r|g|_r$. This proves $|hg|_r= |h|_r|g|_r$.
\end{proof}
Given any $h\in \Acal_R[a,b]-\{0\}$ the Newton polygon is defined to be the function $n_h : (\log a,\log b] \to \R$ given by 
$$
\pi_h(x)=\log |h|_{\exp x}=\max_{n\in \Z} (nx+\log |c_n|)
$$ 
(if $c_n=0$ we take $\log |c_n|=-\infty$). Here we list some of the basic properties of Newton polygons. For the convenience of the reader we give a proof.

\begin{lemma} Let $h,g\in \Acal_R[a,b]-\{0\}$. The function $\pi_h(x)$ is continuous, convex, and piecewise linear with integer slopes. Moreover, we have $\pi_{h(1/t)}(x)=\pi_{h(t)}(-x)$ and $\pi_{hg}(x)=\pi_h(x)+\pi_g(x)$.
\end{lemma}
\begin{proof}
That $\pi_h(x)$ is continuous, convex and piecewise linear with integer slopes is clear from the expression $\pi_h(x)=\max_{n\in \Z} (nx+\log |c_n|)$. To show $\pi_{h(1/t)}(x)=\pi_{h(t)}(-x)$, expand $h(t)$ and $h(1/t)$ as series 
$$
h(t)=\sum_{n\in\Z}c_nt^n,\quad h(1/t)=\sum a_nt^n
$$
and observe that $a_n=c_{-n}$. Hence 
$$
\pi_{h(1/t)}(x)=\max_{n}(nx+\log |a_n|)=\max_{n}((-n)x+\log |c_n|)=\pi_{h(t)}(x).
$$
Finally,  $\pi_{hg}(x)=\pi_h(x)+\pi_g(x)$ follows from $\pi_h(x)=\log |h|_{\exp x}$ and the fact that $h$ is an absolute value.
\end{proof}

\subsection{Analytic functions on $\A^1$ and $\G_m$}

Define $\Acal_R(\A^1)=\cap_{n>0}\Acal_R[0,n]$ and $\Acal_R(\G_m)=\cap_{n>0}\Acal[1/n,n]$. We think about $\Acal_R(\A^1)$ and $\Acal_R(\G_m)$ as rings of analytic functions on the affine line and the punctured affine line (hence the notation involving $\A^1$ and $\G_m$) although we will not use any particular property of $\A^1$ and $\G_m$ as geometric objects.

Observe that $\Acal_R(\A^1)$ is the same $\Acal_R$ defined in the introduction (except from the fact that in this section we assume that $R$ is complete), while $\Acal_R(\G_m)$ is the ring of power series $\sum c_nt^n$ with $c_nr^{|n|}\to 0$ as $|n|\to \infty$ for every $r\in \R^+$. In both cases, given a non-zero $h$ one obtains a Newton polygon $\pi_h$ defined on all $\R$.

The following lemma is well-known (for instance, it follows from Lemma 4.5.2 in \cite{Berkovich}).

\begin{lemma}
The group of units of $\Acal_R(\A^1)$ is $\Acal_R(\A^1)^\times = R^\times$. 
\end{lemma}

In $\Acal_R(\G_m)$ we have the automorphism $\tau\in\Aut_R\Acal_R(\G_m)$ defined by the substitution $t\mapsto t^{-1}$. For analyzing this automorphism we need an elementary fact.

\begin{lemma} If the function $f:\R\to \R$ is convex and odd, then $f$ must be of the form $f(x)=m x$ for some $m\in\R$.
\end{lemma}

Then we have:

\begin{lemma}\label{LemmaNormGm}
Let $h\in \Acal_R(\G_m)$. Suppose that $\tau(h)\cdot h=1$. Then $h=\pm t^n$ for some integer $n$.
\end{lemma} 
\begin{proof} We have $h(t)h(t^{-1})=1$. Thus, for all $x\in \R$ we get
$$
0=\pi_{1}(x)=\pi_{h(t)h(t^{-1})}(x)=\pi_{h(t)}(x)+\pi_{h(t^{-1})}(x)=\pi_{h(t)}(x) + \pi_{h(t)}(-x)
$$
and hence $\pi_h(x)$ is an odd function. But it is also convex, so it is of the form $\pi_h(x)= nx$ for some $n\in \R$. Recalling that $\pi_h(x)$ is piecewise linear with integer slopes we see that $n\in \Z$, and therefore $h=a_nt^n$ for some $n\in \Z$. Using the relation $\tau(h)\cdot h=1$ again, we obtain $a_n^2=1$, hence $a_n\in\{-1,1\}$.
\end{proof}


\subsection{Analytic functions in two variables} In this section we assume that $R$ is a complete \emph{field}. The ring of analytic functions on the variables $t,u$ with coefficients in $R$, denoted by $R\{t,u\}$, is defined to be the ring of all power series
$$
\sum_{m,n\ge 0} c_{mn}t^m u^n
$$
with coefficients in $R$ such that $\lim_{m+n\to \infty}|c_{mn}|r^{m+n}=0$ for every $r\in \R^+$ (this is indeed a ring under the usual addition and multiplication of power series, because $R$ is complete). We will be interested in quotients of this ring. 

\begin{lemma}\label{LemmaGmQuot} The rule $t\mapsto t$, $u\mapsto t^{-1}$ extends to an isomorphism of $R$-algebras
$$
\mu: \frac{R\{t,u\}}{(tu-1)R\{t,u\}}\to \Acal_R(\G_m).
$$
\end{lemma}
\begin{proof}
Consider the map $M: R\{t,u\} \to \Acal_R(\G_m)$ given by $M(f(t,u))=f(t,t^{-1})$. We will check that this map is well defined, that it is an $R$-algebra morphism, it is surjective and its kernel is $(tu-1)R\{t,u\}$. After all this is checked, $\mu$ will be the map induced by $M$ and the result will be proved.

\emph{$M$ is well defined:} Let $f(t,u)=\sum_{m,n\ge 0} c_{mn}t^mu^n\in R\{t,u\}$. Computing $M(f)$ formally we find $M(f)=\sum_{d} a_dt^d$ where
$$
a_d=\sum_{m-n=d} c_{mn},\quad d\in \Z.
$$
It suffices to show that each $a_d$ is a convergent series (that is, $a_d\in R$) and that for any given $r>0$ we have 

\begin{equation}\label{limit}
\lim _{|d|\to \infty} |a_d|r^d=0.
\end{equation}

For fixed $d$, the pairs $(m,n)$ with $m-n=d$ satisfy $m+n\to \infty$, hence $|c_{mn}|\to 0$ as $m-n=d$ because $f\in R\{t,u\}$. Therefore $a_d\in R$.

Given $d$ let $(m_d,n_d)$ be a pair with $m_d-n_d=d$ that maximizes the quantity $|c_{mn}|$. Note that $m_d+n_d\ge |d|$ because $m_d,n_d\ge 0$ and $m_d-n_d=d$. Therefore $|a_d|\le \max_{m-n=d}|c_{mn}|=|c_{m_dn_d}|\to 0$ as $|d|\to\infty$ (again, because $f\in R\{t,u\}$). This proves \eqref{limit} for $r=1$. In particular, there is $K>0$ such that $|a_d|<K$ for all $d$.

We now prove \eqref{limit} when $r>1$ (the case $0<r<1$ is similar). If $d\to -\infty$ then $|a_d|r^d\le Kr^d\to 0$. On the other hand, if $d\to +\infty$ then  
$$
|a_d|r^d\le |c_{m_dn_d}|r^d\le |c_{m_dn_d}| r^{m_d+n_d}\to 0
$$
because $m_d+n_d\ge |d|=d$ and $f\in R\{t,u\}$. This proves \eqref{limit} for all $r>0$.

\emph{$M$ is an $R$-algebra morphism:} $R$-linearity is clear, so we only need to check that $M(fg)=M(f)M(g)$ for all $f,g\in  R\{t,u\}$. Write
$f(t,u)=\sum_{m,n\ge 0} b_{mn}t^mu^n$ and $f(t,u)=\sum_{m,n\ge 0} c_{mn}t^mu^n$. Using the formula for the coefficients of $M(f)$ proved in the previous item, we find $M(fg)=\sum_{d} a_d t^d$ where
$$
a_d=\sum_{m-n=d} \sum_{(i,j)+(r,s)=(m,n)}b_{ij}c_{rs} 
$$
and $M(f)M(g)=\sum_d a'_d t^d$ where
$$
a'_d=\sum_{d_1+d_2=d} \left(\sum_{i-j=d_1}b_{ij}\right)\left(\sum_{r-s=d_2}c_{rs}\right).
$$
The series for both $a_d$ and $a'_d$ are absolutely convergent by the computations in the previous item (namely, the verification that $M$ is well-defined) and because $R\{u,t\}$ and $\Acal(\G_m)$ are closed under multiplication. Thus we can rearrange the summands to conclude $a_d=a'_d$ for all $d$, which proves $M(fg)=M(f)M(g)$.  

\emph{$M$ is surjective:} Given $h\in \Acal_R(\G_m)$ we can separate terms with negative exponent on $t$, and terms with non-negative exponent on $t$, to conclude that there are $g_1,g_2\in\Acal_R(\A^1)$ such that $h(t)=g_1(t)+g_2(t^{-1})$. Let $f(t,u)=g_1(t)+g_2(u)$, then $f\in R\{t,u\}$ and $M(f)=h$.

\emph{The kernel of $M$ is $(tu-1)R\{t,u\}$:} Clearly $(tu-1)R\{t,u\}\subseteq \ker(M)$, and we will show the converse inclusion. Let $f=\sum_{m,n\ge 0} c_{mn}t^mu^n\in \ker (M)$. For integers $m,n$ we define $\gamma_{m,n}\in R$ recursively as follows:
\begin{itemize}
\item $\gamma_{mn}=0$ if $m<0$ or $n<0$.
\item $\gamma_{mn}=\gamma_{m-1,n-1}-c_{mn}$ for $m,n\ge 0$. 
\end{itemize}
Considering the points $(m,n)\in\Z^2$ it can be seen that these conditions uniquely define elements $\gamma_{mn}\in R$, which are non-zero only if $m,n\ge 0$. Indeed, if $m,n\ge 0$ we can iterate the second condition to obtain
$$
\gamma_{mn}=-\sum_{j=0}^{\min\{m,n\}}c_{m-j,n-j}.
$$
Moreover, since $M(f)=0$ we know $\sum_{m-n=d}c_{mn}=0$ for all $d$, so we deduce
$$
\gamma_{mn}=\sum_{j\ge 1}c_{m+j,n+j}.
$$

Consider the formal power series $F=\sum \gamma_{mn}t^mu^n$. We claim that $F\in R\{t,u\}$. Indeed, we have
$$
|\gamma_{mn}|=\left|\sum_{j\ge 1}c_{m+j,n+j}\right|\le \max_{j\ge 1}|c_{m+j,n+j}|.
$$
Therefore, for any given large $r$ (actually $r\ge 1$ suffices) we get
$$
|\gamma_{mn}|r^{m+n}\le \left(\max_{j\ge 1}|c_{m+j,n+j}|\right)r^{m+n}\le \max_{v+w\ge m+n}|c_{vw}|r^{v+w}\to 0
$$
as $m+n\to \infty$, because $\lim_{v+w\to \infty} |c_{vw}|r^{v+w}=0$. This shows $F\in R\{t,u\}$.

From the definition of the elements $\gamma_{mn}$ we find $f=(tu-1)F$ and therefore $f\in (tu-1)R\{t,u\}$.

This concludes the proof.
\end{proof}

There is a canonical injection of the polynomial ring $\Acal_R[u]$ into $R\{t,u\}$. This is preserved under suitable quotients.

\begin{lemma}\label{LemmaQuot}
Let $F\in \Acal_R[u]$ be a monic polynomial which is non-constant in $u$. The natural map 
$$
q: \Acal_R[u]\to R\{t,u\}/FR\{t,u\}
$$
induces an injective $\Acal_R$-algebra morphism 
$$
\Acal_R[u]/F\Acal_R[u]\to R\{t,u\}/FR\{t,u\}.
$$
\end{lemma}
\begin{proof}
One has to show that $\ker (q) = F\Acal_R[u]$, or equivalently, that 
$$
\Acal_R[u]\cap F R\{t,u\}=F\Acal_R[u]. 
$$
The inclusion $\Acal_R[u]\cap F R\{t,u\}\supseteq F\Acal_R[u]$ is clear. For the reciprocal inclusion, let $h\in R\{t,u\}$ be such that $Fh=G\in \Acal_R[u]$. Let $Q,H\in \Acal_R[u]$ be such that $G=FQ+H$ and $\deg H<\deg F$ (they exist because $F$ is monic). We claim that $h=Q$. For otherwise, we have $H=(h-Q)F=t^rgF$ for some $r\ge 0$ and some $g\in R\{t,u\}$ with $g(0,u)\ne 0$. We may assume that $r=0$ (after dividing by $t^r$ the coefficients of $H$ if necessary) and therefore substituting $t=0$ in $H=gF$ we reach a contradiction ($F$ is monic and we can compare degrees in $u$). Therefore $h=Q\in \Acal_R[u]$. 
\end{proof}

\section{Analytic solutions of Pell equations}\label{SecPell}

\subsection{Analytic solutions are actually polynomials}
 
In \cite{Denef0}, Denef studied the solutions of the Pell equation
$$
X^2-(t^2-1)Y^2=1
$$
over a polynomial ring $R[t]$ when $R$ is an integral domain of characteristic zero. He proved that the only solutions in $R[t]$ are of the form $(\pm x_n,y_n)$ where $x_n,y_n\in R[t]$ are defined by
\begin{equation}\label{Sol0}
x_n + y_n \sqrt{t^2-1}=(t+\sqrt{t^2-1})^n, \quad n\in \Z.
\end{equation}
Then in \cite{Denefp}, Denef proved that the same is true whenever $R$ is an integral domain of characteristic $p>2$. Moreover, when $\cha R=2$ he proved that the only solutions in $R[t]$ of
$$
X^2+tXY+Y^2=1
$$
are given by the pairs $(\pm x_n,y_n)$ defined by
\begin{equation}\label{Sol2}
x_n+y_n\alpha = \alpha^n \quad n\in \Z
\end{equation}
where $\alpha$ is a root of $Z^2+tZ+1=0$. In this section we establish the following theorem.
\begin{theorem}\label{ThmPell}
Let $R$ be an integral domain endowed with a non-Archimedean absolute value $|\cdot |$. If $\cha R\ne 2$, then the only solutions in $\Acal_R$ of the Pell equation 
\begin{equation}\label{Pell0}
X^2-(t^2-1)Y^2=1
\end{equation}
are the pairs of polynomials $(\pm x_n,y_n)$ defined by \eqref{Sol0}.  If $\cha R= 2$, then the only solutions in $\Acal_R$ of the Pell equation 
\begin{equation}\label{Pell2}
X^2+ tXY + Y^2=1
\end{equation}
are the pairs of polynomials $(x_n,y_n)$ defined by \eqref{Sol2}. In all cases, the polynomials $x_n$ and $y_n$ have coefficients in the sub-ring of $R$ generated by $1$ (that is, $\Z$ or $\F_p$ depending on the characteristic).
\end{theorem}
Of course, it is not hard to see that $x_n,y_n$ indeed give solutions and they are polynomials with coefficients in $\Z$ or $\F_p$ according to the characteristic (this is already in Denef's work); the difficult part is showing that these are the \emph{only} solutions in $\Acal_R$.

We remark that the case $R=\C_p$ (hence $\Acal_R=\Hcal_p$) was first established by Lipshitz and Pheidas in \cite{LipPhe}. Our proof is different and makes no use of differentiation, hence, it works even in positive characteristic. On the other hand, the analogue for complex entire functions is easily seen to be false, see \cite{PhZaSurvey}.

Before proving Theorem \ref{ThmPell}, let us state what we can obtain from it if we combine it with Denef's results on polynomial solutions of Pell equations. For simplicity, we only consider the case of characteristic $p>2$; the case of characteristic $2$ is similar and can be deduced from Lemma 3.1 in \cite{Denefp}

\begin{corollary} \label{DenefPell}
Assume that $\cha R=p>2$. The polynomials $x_n,y_n$ defined by \eqref{Sol0} satisfy the following:
\begin{enumerate}
\item $x_n$ has degree $|n|$ for $n\in\Z$, and $y_n$ has degree $|n|-1$ for $n\ne 0$.
\item All solutions of $X^2-(t^2-1)Y^2=1$ in $\Acal_R$ are given by the pairs $(x_n,y_n)$ and $(-x_n,y_n)$ for $n\in\Z$.
\item $x_{m+n}=x_mx_n+(t^2-1)y_my_n$ and $y_{m+n}=x_my_n+x_ny_m$.
\item $m$ divides $n$ in $\Z$ if and only if $y_m$ divides $y_n$ in $\Acal_R$.
\item For $r\ge 0$ we have $x_{mp^r}=x_m^{p^r}$, in particular $x_{p^r}=t^{p^r}$.
\item $x_m(t+1)=x_m(t)+1$ if and only if $m=\pm p^r$ for some $r\ge 0$.
\item $x_m\equiv 1 \mod (t-1)$, congruence in $\Acal_R$.
\end{enumerate}
\end{corollary}
\begin{proof}
After Theorem \ref{ThmPell}, this is a consequence of Lemma 2.1 \cite{Denefp}. The only fact that requires some attention is the following: if $0\ne P\in R[t]$ and $h\in\Acal_R$ divides $P$ in $\Acal_R$, then actually $h\in R[t]$ (this is needed in items 4 and 7). Indeed, let $d=\deg P$ and suppose that $P=hg\in R[t]$, then as  $x\to +\infty$ we have
$$
\pi_h(x) + \pi_g(x)= \pi_P(x)< dx +C
$$
for some constant $C$. Both $\pi_h(x)$ and $\pi_g(x)$ are bounded from below as $x\to +\infty$, hence they have at most linear growth. Therefore $h$ and $g$ are polynomials.
\end{proof}


\subsection{Characteristic different from $2$} 

Let us prove Theorem \ref{ThmPell} when $\cha R\ne 2$. We can assume that $R$ is a complete field by enlarging $R$ if necessary.  Put $\alpha = \sqrt{t^2-1}$. 

\begin{lemma}
Let $\Mcal_R$ be the fraction field of $\Acal_R$. Then $\alpha \notin \Mcal_R$, hence $\alpha$ is algebraic of degree $2$ over $\Mcal_R$.
\end{lemma}
\begin{proof} We can assume that $R$ is an algebraically closed complete field. Assume towards a contradiction, that $\alpha\in \Mcal_R$. Then actually $\alpha \in \Acal_R$ because $\alpha^2=t^2-1$ has no poles. We have $1=t^2-\alpha^2=(t+\alpha)(t-\alpha)$ with both $t\pm \alpha\in \Acal_R$, hence in $R$ because $\Acal_R^\times =R^\times$, from which we conclude that $2t=(t+\alpha)+(t-\alpha)\in R$, a contradiction.
\end{proof}

In particular, we have an automorphism $\sigma \in \Aut_{\Acal_R} \Acal_R[\alpha]$ defined by $\sigma(\alpha)=-\alpha$ (that is, interchanging the roots of $Z^2-t^2+1$). If $(a,b)\in\Acal_R^2$ is a solution of \eqref{Pell0} then $\omega =a+\alpha b\in \Acal_R[\alpha]$ satisfies $\omega\sigma(\omega)=1$. 

The natural map
$$
i:\Acal_R[\alpha]= \frac{\Acal_R[u]}{(t^2-u^2-1)\Acal_R[u]}\to B:=\frac{R\{t,u\}}{(t^2-u^2-1)R\{t,u\}}
$$
is injective by Lemma \ref{LemmaQuot}. Under this inclusion, the map $\sigma$ is a restriction of the automorphism $\epsilon\in\Aut_R B$ defined by the substitution $u\mapsto -u$. 
\begin{lemma}\label{LemmaIso0} The rule $t\mapsto (z+w)/2$, $u\mapsto (z-w)/2$ extends (by substitution) to an $R$-isomorphism
$$
\delta: B \to \frac{R\{z,w\}}{(zw-1)R\{z,w\}}
$$
whose inverse $\delta'$ is determined by $z\mapsto t+u$, $w\mapsto t-u$.
\end{lemma}
\begin{proof} To check that $\delta$ and $\delta'$ take convergent power series to convergent power series is a straightforward application of the strong triangle inequality. One checks that  $\delta$ and $\delta'$  are well defined by substituting the variables in $t^2-u^2-1=(t+u)(t-u)-1$ and in $zw-1$. That $\delta$ and $\delta'$ are inverses of each other is clear.
\end{proof}
By Lemma \ref{LemmaGmQuot}, the map $\delta'$ induces an $R$-isomorphism
$$
d':\Acal_R(\G_m)\to B
$$
given by $z\mapsto t+u$, $z^{-1}\mapsto t-u$, with inverse denoted by $d$. Under these isomorphisms, the automorphism $\epsilon$ on $B$ corresponds to $\tau\in\Aut_R\Acal_R(\G_m)$ defined by $z\mapsto z^{-1}$, namely $\tau d= d\epsilon$. 

Observe that 
$$
d i(\omega)\cdot \tau d i(\omega)=d i(\omega)\cdot d \epsilon i(\omega)=d i(\omega)\cdot  d i \sigma (\omega)=di(\omega\sigma(\omega))=di(1)=1.
$$
Lemma \ref{LemmaNormGm} implies $di(\omega)=\pm z^n$ for some integer $n$, hence $\omega=\pm (t+\alpha)^n$. This concludes the proof of Theorem \ref{ThmPell} when $\cha R\ne 2$


\subsection{Characteristic $2$} 

Now we prove Theorem \ref{ThmPell} when $\cha R= 2$. Again we can assume that $R$ is a complete field.  This time we define $\alpha$ as a solution of $Z^2+tZ+1=0$. One can show that $\alpha\notin \Acal_R$ and similarly as before we let $\sigma \in \Aut_{\Acal_R} \Acal_R[\alpha]$ be defined by $\sigma(\alpha)=t+\alpha$ (that is, we interchange the roots of $Z^2+tZ+1$). 

With this setup, the proof goes exactly as in the previous section, except that the ring $B$ should be defined as
$$
B=\frac{R\{t,u\}}{(u^2+tu+1)R\{t,u\}}
$$
and Lemma \ref{LemmaIso0} should be replaced by 

\begin{lemma} The rule $t\mapsto z+w$, $u\mapsto z$ extends (by substitution) to an $R$-isomorphism
$$
\delta: B \to \frac{R\{z,w\}}{(zw-1)R\{z,w\}}
$$
whose inverse $\delta'$ is determined by $z\mapsto u$, $w\mapsto t+u$.
\end{lemma}

We leave the remaining details to the reader.


\section{The case of fixed characteristic}\label{SecFixedChar}

 In this section we discuss how to interpret in a positive existential way the ring of rational integers in the $L_t$-structure $\Acal_R$ when $\cha R=p\ge  0$ is given. That is, the interpretation in this section will not be uniform in the characteristic. However, this is enough for proving Theorem \ref{MainThm}, and indeed it is a relevant step in our final goal of interpreting the integers uniformly as we will see in Section \ref{SecFinal}.

Before going to the positive characteristic case, first we prove

\begin{theorem}\label{ThmChar0}
The ring of integers $(\Z;0,1,+,\cdot,=)$ is uniformly positive existentially interpretable in the class of $L_t$-structures $\Acal_R$ as $R$ ranges over the integral domains of characteristic zero with a non-Archimedean absolute value. 
\end{theorem}
\begin{proof}
By Theorem \ref{ThmPell}, the argument of \cite{Denef0} applies directly. Indeed, this is the same way to proceed as in \cite{LipPhe}, once one knows that the analytic solutions of the Pell equation $X^2-(t^2-1)Y^2=1$ are exactly the same polynomial solutions from \cite{Denef0}. In a nutshell, one recovers the integers as the values $y_n(1)=n$ (here, $(\pm x_n,y_n)$ are the solutions of the Pell equation).
\end{proof}

When the characteristic is positive, we would like to say that thanks to Theorem \ref{ThmPell} the argument is the same as in \cite{Denefp}, but actually we can't. The reason is that we only considered the Pell equation $X^2-(t^2-1)Y^2=1$ (when $p\ne 2$), but for the argument in \cite{Denefp} to work one would need that the analytic solutions of $X^2-(a^2-1)Y^2=1$ are polynomials whenever $a$ is a non-constant polynomial. We do not know if the latter is true and our techniques from Section \ref{SecPell} do not apply directly to this more general Pell equation. We can circumvent this difficulty by invoking the general results of \cite{PheidasII}.

\begin{theorem}\label{ThmCharp}
Let $p$ be a prime. The ring $(\Z;0,1,+,\cdot,=)$ is uniformly positive existentially interpretable in the class of $L_t$-structures $\Acal_R$ as $R$ ranges over the integral domains of characteristic $p$ with a non-Archimedean absolute value.
\end{theorem}
\begin{proof}
It suffices to prove the theorem for the semi-ring $(\N;0,1,+,\cdot,=)$ instead of the ring $(\Z;0,1,+,\cdot,=)$. After analyzing the arguments in \cite{PheidasII}, we see that it suffices to find a positive existential $L_t$-definition of the set
$$
P=\{t,t^2,t^3,t^4,\ldots\} \subseteq \Acal_R
$$
which is uniform on the ring $R$, although it can depend on $p$ (indeed, this is immediate from the main theorem of \cite{PheidasII} if uniformity on $R$ is ignored, but actually for fixed $p$ the formulas in \cite{PheidasII} do not depend on $R$). 

Fix $p>0$. Let us first give a positive existential $L_t$-definition of the set
$$
F=\{t,t^{p^2},t^{p^3},t^{p^4},\ldots\} \subseteq \Acal_R
$$
which is uniform on the ring $R$. 

Assume $p>2$. Consider the formula
$$
\begin{aligned}
\phi(f):\quad \exists y,h,u,v,g, &(f^2-(t^2-1)y^2=1)\wedge (f=1+(t-1)h)\wedge\\
 &(u^2-((t+1)^2-1)v^2=1)\wedge (u=1+tg)\wedge (u=f+1).
\end{aligned}
$$
Theorem \ref{ThmPell} applies to the Pell equation $X^2-((t+1)^2-1)Y^2=1$ (that is, this equation has only polynomial solutions) because the substitution $t\mapsto t+1$ defines an automorphism of $\Acal_{\tilde{R}}$ where $\tilde{R}$ is the completion of $R$. Therefore, we can apply Lemma 2.1 \cite{Denefp} in the special case $a=t$ (specially items 5, 6 and 7) to conclude that the formula $\phi(f)$ defines the set $F$. This gives the desired definition for $F$ in the case $p>2$. When $p=2$ we proceed similarly, but we use Lemma 3.1 from \cite{Denefp} instead (specially items 5 and 6).

Finally, we claim that the formula
$$
\psi(f):\quad \exists h, \phi(h)\wedge (f|h)\wedge (t|f)\wedge (t-1|f-1)
$$
defines $P$ in $\Acal_R$ (here, the symbol $|$ stands for divisibility in $\Acal_R$, which is $L_t$-definable in a uniform positive existential way). Indeed, $\psi(f)$ holds if and only if
\begin{itemize}
\item $h=t^{p^n}$ for some $n \ge 1$, and
\item $f$ divides $t^{p^n}$ in $\Acal_R$, 
\item $t$ divides $f$ in $\Acal_R$, and
\item $f(1)=1$ (this makes sense even if $R$ is not complete, because any $f$ satisfying the second item is a polynomial; see the proof of Corollary \ref{DenefPell}),
\end{itemize}
but these items hold if and only if $f\in P$.
\end{proof}

As the reader can check, the formulas $\phi(f)$ and $\psi(f)$ do not depend on the characteristic when $p>2$. The reason for the previous theorem to be non-uniform in the characteristic is due to the use of Pheidas' results from \cite{PheidasII}. We do not know how to make Pheidas' theorem uniform, and a different approach is necessary. Instead, we will resurrect Denef's approach from \cite{Denefp} in this context, replacing the use of the more general Pell equation $X^2-(a^2-1)Y^2=1$ by an application of B\"uchi's problem.

Although we do not have uniformity in the characteristic at this point, Theorem \ref{ThmChar0} and Theorem \ref{ThmCharp} already prove Theorem \ref{MainThm} thanks to Matijasevich's result \cite{Matijasevich}.


\section{B\"uchi's problem and the $p^r$-th power relation} \label{SecBuchi}

Let us recall the following theorem from \cite{Corr} (see also \cite{PheVid}, \cite{ShlVid} and \cite{PastenWang}). It gives a solution to \emph{B\"uchi's $n$ squares problem} for polynomial rings in positive characteristic.

\begin{theorem}\label{ThmBuchi}
Let $k$ be an algebraically closed field of characteristic $p\ge 17$. Suppose that $u_1,\ldots, u_{17}$ are squares in the polynomial ring $k[t]$ such that at least one of them is not in $k$, and such that for each $1\le n\le 15$ we have $u_{n+2}-2u_{n+1}+u_n=2$. Then there is $v\in k[t]$ and $r\ge 0$ such that for each $n$
$$
u_n=(n+v)^{p^r+1}.
$$
\end{theorem}

B\"uchi's problem was originally formulated by J. B\"uchi on $\Z$, and remains open. It states the following:
\begin{problem} There is a constant $M$ with the following property:

Suppose that $u_1,\ldots, u_M$ are integer squares such that for each $1\le n\le M-2$ we have $u_{n+2}-2u_{n+1}+u_n=2$. Then there is $v\in \Z$ such that $u_n=(n+v)^2$ for each $n$.
\end{problem}
 The original intended application was to obtain a strong form of the negative solution of Hilbert's tenth problem: undecidability for the problem of simultaneous representation of integers by diagonal quadratic forms. See \cite{collected} or \cite{MazurBuchi} for details on B\"uchi's original problem. Subsequent analogues of this problem have found similar applications for diagonal quadratic forms, see for instance \cite{Survey}. The application of Theorem \ref{ThmBuchi} that we give here, however, is different; it is closer to the application given in \cite{PPV}.

Consider the following relation on $\Acal_R$, for $R$ of characteristic $p>0$: we say that $u\ge_p v$ if there is $r\ge 0$ such that $u=v^{p^r}$.

We will use Theorem \ref{ThmBuchi} for expressing the relation $u\ge_p v$ uniformly in the language $L_t$, in very much the same way as in \cite{PPV} (although in \cite{PPV} it was used a symmetrized version of $\ge_p$). Actually we are not able to fully define $\ge_p$ in $\Acal_R$ because we do not know if the analogue of Theorem \ref{ThmBuchi} holds for $\Acal_R$. However, we will be able to partially define $\ge_p$ in a way that is sufficient for our purposes (using the fact that $R[t]\subseteq \Acal_R$). 

\begin{theorem}\label{DefBuchi}
There is a positive existential $L_t$-formula $\beta(x,y)$ with the following property:

Let $R$ be an integral domain of characteristic $p\ge 17$ endowed with a non-Archimedean absolute value. Let $f,g\in \Acal_R$. If actually $f,g\in R[t]$ then the following two items are equivalent
\begin{itemize}
\item[(a)] $\Acal_R$ satisfies $\beta(f,g)$, 
\item[(b)] $f\ge_p g$.
\end{itemize}
\end{theorem}
\begin{proof}
We claim that the formula 
$$
\begin{aligned}
\beta'(x,y):\quad \exists u_1,\ldots, u_{17},z,&\bigwedge_{n=1}^{15}(u_{n+2}-2u_{n+1}+u_n=2)\wedge (xy=u_1)\wedge \\
&(x+y=u_2-u_1-1)\wedge (x=yz)
\end{aligned}
$$
has the next property: for $\cha R=p\ge 17$, if $f,g\in R[t]$ and $f$ or $g$ is non-constant, then  $\Acal_R$ satisfies $\beta'(f,g)$ if and only if $f\ge_p g$.

Suppose that $f\ge_p g$, say, $f=g^{p^r}$. Taking $u_n=(n-1+g)^{p^r+1}$ and $z=g^{p^r-1}$ we see that $\Acal_R$ satisfies $\beta'(f,g)$. 

Conversely, if $\Acal_R$ satisfies $\beta'(f,g)$ then $u_1=fg$ and $u_2=f+g+u_1+1$ belong to $R[t]$. Solving the second order recurrence imposed by $\bigwedge_{n=1}^{15}(u_{n+2}-2u_{n+1}+u_n=2)$ we can express each $u_n$ in terms of $u_1$ and $u_2$ showing that for each $n$ we have $u_n\in R[t]$. Moreover, since $f$ or $g$ has positive degree so does $u_1$ or $u_2$, hence we can apply Theorem \ref{ThmBuchi} to conclude that there is $v\in k[t]$ and $r\ge 0$ such that for each $n$ we have $u_n=(n+v)^{p^r+1}$ (here, $k$ is the algebraic closure of the fraction field of $R$). One can directly check that the pairs $(v+1,(v+1)^{p^r})$ and $((v+1)^{p^r},v+1)$ are solutions of the system of equations
$$
\begin{cases}
XY&=u_1\\
X+Y&=u_2-u_1-1
\end{cases}
$$
which has exactly two solutions counting multiplicity, hence, it has no other solutions (the case $r=0$ requires to check that there is just one double solution, which is true because the discriminant of the polynomial 
$$
Z^2-(u_2-u_1-1)Z+u_1=Z^2-((v+2)^2-(v+1)^2-1)Z+(v+1)^2
$$ 
is $0$). However, the pairs $(f,g)$ and $(g,f)$ are also solutions because $\Acal_R$ satisfies $\beta'(f,g)$, hence, $(f,g)=(v+1,(v+1)^{p^r})$ or $(f,g)=((v+1)^{p^r},v+1)$. The clause $(f=zg)$ implies $g|f$ and therefore we conclude $(f,g)=((v+1)^{p^r},v+1)$. Thus $f\ge_p g$.

Finally, we take
$$
\beta(x,y):\quad \exists u,v, (u^2-(t^2-1)v^2=1)\wedge \beta'(u,t)\wedge \beta'(ux,ty)\wedge \beta'(x,y).
$$
Let us check that this formula works. Let $f,g\in R[t]$. 

If (b) holds, say $f=g^{p^r}$, then we take $u=t^{p^r}$ and $v=(t^2-1)^{(p^r-1)/2}$. It is straightforward to check that $\Acal_R$ satisfies $\beta(f,g)$ with these choices of $u,v$ thanks to our previous work on $\beta'(x,y)$, hence (a) holds.

Conversely, if (a) holds then Theorem \ref{ThmPell} implies that $u,v\in R[t]$. Since $u,t\in R[t]$ and $t$ has positive degree, the fact that $\Acal_R$ satisfies $\beta'(u,t)$ implies that $u=t^{p^r}$ for some $r\ge 0$. We observe that $uf$ and $tg$ belong to $R[t]$. If both $f,g$ are $0$ then $f\ge_p g$ immediately, so we can assume that $f$ or $g$ is not $0$. Then at least one of $uf$ or $tg$ has positive degree, and the fact that $\Acal_R$ satisfies $\beta'(uf,tg)$ implies that $uf=(tg)^{p^s}$ for some $s\ge 0$. Hence $t^{p^r}f=t^{p^s}g^{p^s}$ and we want to conclude that (b) holds. If both $f,g$ belong to $R$ then $r=s$ and $f=g^{p^r}$. On the other hand, if $f$ or $g$ has positive degree then the fact that they are polynomials and $\Acal_R$ satisfies $\beta'(f,g)$ implies that $f\ge_p g$. Therefore (b) holds.
\end{proof}


\section{Uniform interpretation of the integers}\label{SecFinal}

\subsection{Outline of the proof of Theorem \ref{MainUnif}} 

In this section we prove our main result, Theorem \ref{MainUnif}. For this it will be useful to introduce a collection of auxiliary structures.

Consider the language $L^*=\{0,1,+,|,|^*,\ne\}$ and for each prime $p$ we define the the $L^*$-structure
$$
\Zfrak_p=(\Z; 0,1,+,|,|^p,\ne)$$ 
where $n|^p m $ means that there is $r \ge 0$ such that $m=\pm p^rn$. 

The proof of Theorem \ref{MainUnif} is done in three steps.
\begin{itemize}
\item[Step 1.] For $p\ge 17$, we give a uniform positive existential interpretation of $\Zfrak_p$ on the class of $L_t$ structures $\Acal_R$ when $\cha R=p$. The key point is that the formulas will be \emph{independent of $p$}. 
\item[Step 2.] Using results of \cite{PPV}, we will deduce that the ring $\Z$ is uniformly positive existentially interpretable in the class of $L_t$-structures $\Acal_R$, for $\cha R\ge 17$.

\item[Step 3.] Finally, we cover the finitely many remaining characteristics using the results of Section \ref{SecFixedChar}.

\end{itemize}

The rest of this section is devoted to these steps.

\subsection{Step 1} We need the next result.

\begin{lemma}\label{Lemmane}
There is a positive existential $L_t$-formula $\nu(x)$ with the following property:

Let $R$ be an integral domain of characteristic $p>0$ with a non-Archimedean absolute value. For all $f\in \F_p[t]\subseteq \Acal_R$ we have that: $\Acal_R$ satisfies $\nu(f)$ if and only if $f\ne 0$.
\end{lemma}
\begin{proof}
We claim that the formula
$$
\nu(x): \exists a,b,c,(ta+1)((t-1)b+1)=fc
$$
works. Let us remark that this formula is known to work in the polynomial case, but here the variables under quantifiers range over $\Acal_R$. 

Assume first that $\Acal_R$ satisfies $\nu(f)$, then $ta\ne -1$ and $(t-1)b\ne -1$ because $t$ and $t-1$ are not invertible in $\Acal_R$ (this can be seen in a number of ways, for instance, looking at Newton polygons). Hence $fc\ne 0$ because $\Acal_R$ is an integral domain, and we conclude that $f\ne 0$.

Conversely, assume that $f\ne 0$, where $f\in \F_p[t]\subseteq \Acal_R$. Let us factor $f$ in $\F_p[t]$ as $f=t^\alpha (t-1)^\beta\Gamma$ where $t,t-1$ do not divide $\Gamma$ in $\F_p[t]$. Put $F=(t-1)^\beta \Gamma$ and $G=t^\alpha \Gamma$. Since $(t,F)=1$ in $\F_p[t]$ we see that there are $u,v\in \F_p[t]$ such that $tu+Fv=1$. Similarly, there are $r,s\in \F_p[t]$ such that $(t-1)s+Gr=1$. From these equations we deduce $(-tu+1)(-(t-1)s+1)=FvGr=f\Gamma vr$, so we can take $a=-u$, $b=-s$ and $c=\Gamma vr$ which belong to $\F_p[t]\subseteq \Acal_R$. 
\end{proof}

The goal of the present step is achieved by the following result.

\begin{lemma}\label{MainLemma}
There is a set of positive existential $L_t$-formulas $\Phi=\{\phi_s\}_{s\in L^*\cup \{ L^*\}}$ with the following property:

Let $p\ge 17$. The $L^*$-structure $\Zfrak_p$ is uniformly interpretable by $\Phi$ in the class of $L_t$ structures $\Acal_R$ as $R$ ranges over all the integral domains of characteristic $p$ endowed with a non-Archimedean absolute value.
\end{lemma} 
\begin{proof}
First we list the formulas of $\Phi$:
$$
\begin{aligned}
\phi_{L^*}(x,y):\quad &\exists z, (x^2-(t^2-1)y^2=1)\wedge(x=1+(t-1)z)\\
\phi_0(x,y):\quad &x=1\wedge y=0\\
\phi_1(x,y):\quad &x=t\wedge y=1\\
\phi_+(x,y,u,v,f,g):\quad &\phi_{L^*}(x,y)\wedge \phi_{L^*}(u,v)\wedge\\
   &(f=xu+(t^2-1)yv)\wedge(g=xv+yu)\\
\phi_|(x,y,u,v): \quad &\exists z, \phi_{L^*}(x,y)\wedge \phi_{L^*}(u,v)\wedge(v=yz)\\
\phi_{|^*}(x,y,u,v): \quad &\phi_{L^*}(x,y)\wedge \phi_{L^*}(u,v)\wedge \beta(u,x)\\
\phi_{=}(x,y,u,v): \quad &\phi_{L^*}(x,y)\wedge \phi_{L^*}(u,v)\wedge(x=u)\wedge(y=v)\\
\phi_{\ne}(x,y,u,v): \quad &\phi_{L^*}(x,y)\wedge \phi_{L^*}(u,v)\wedge (\nu(x-u)\vee \nu(y-v))
\end{aligned}
$$
with $\beta(x,y)$ as in Theorem \ref{DefBuchi}, and $\nu(x)$ as in the previous lemma. 

Let us check that these formulas work. By Theorem \ref{ThmPell} and Item 7 of Corollary \ref{DenefPell}, we see that for each $R$ as in the statement we have
$$
\phi_{L^*}^{\Acal_R}=\{(x_n,y_n) : n\in \Z, x_n+ y_n\sqrt{t^2-1}=(t+\sqrt{t^2-1})^n\}.
$$
Then we take 
$$
\theta_{\Acal_R}:\phi_{L^*}^{\Acal_R}\to \Z=|\Zfrak_p|,\quad (x_n,y_n)\mapsto n.
$$
Finally, $\phi_0$ and $\phi_1$ work because $(x_0,y_0)=(1,0)$ and $(x_1,y_1)=(t,1)$; $\phi_+$ works by Item 3 of Corollary \ref{DenefPell}; $\phi_|$ works by Item 4 of Corollary \ref{DenefPell}; $\phi_{|^*}$ works by Theorem \ref{DefBuchi} along with Item 5 of Corollary \ref{DenefPell}; $\phi_{=}$ works because $\theta_{\Acal_R}$ is bijective, and $\phi_{\ne}$ works by Lemma \ref{Lemmane} and by the last assertion of Theorem \ref{ThmPell}.
\end{proof}

We remark that the previous proof is the same as the arguments of \cite{Denefp} with respect to $\phi_{L^*}$, $\phi_0$, $\phi_1$, $\phi_+$, $\phi_|$ (of course, after we know Theorem \ref{ThmPell}), but the argument is different for $\phi_{|^*}$. Moreover, for our purposes we also need $\phi_{\ne}$.

\subsection{Step 2} Note that the relation $|_p$ in $\Z$ used in \cite{PPV} is the simetrization of the relation $|^p$ used by us. Hence, $|_p$ can be defined over $L^*$ by 
$$
x|_p y \mbox{ if and only if } \Zfrak_p\mbox{ satisfies } x|^*y\vee y|^*x
$$ 
which gives a uniform positive existential $L^*$-definition. Similarly, the relation $T$ on $\Z$ (that is, $\Z -\{-1,0,1\}$) used in \cite{PPV} can be defined over $L^*$ by
$$
T(x)\mbox{ if and only if }\Zfrak_p\mbox{ satisfies }x+1\ne0\wedge x\ne 0\wedge x\ne 1
$$
which gives a uniform positive existential $L^*$-definition. Therefore Lemma \ref{MainLemma} implies that, for each given $p\ge 17$, the structures $\Dfrak_p=(\Z;0,1,+,|,|_p,=,T)$ from \cite{PPV} are uniformly positive existentially interpretable in the class of $L_t$ structures $\Acal_R$ as $R$ ranges over all the integral domains of characteristic $p\ge 17$ endowed with a non-Archimedean absolute value, by a set of formulas $\Phi'$ which is \emph{independent of $p$}.

On the other hand, the ring $(\Z;0,1,+,\cdot,=)$ is uniformly positive existentially interpretable in the structures $\Dfrak_p$ by Theorem 4.3 from \cite{PPV}. Therefore, by the \emph{transitivity property} of uniform interpretability (cf. Section 3.3 \cite{PPV}) we conclude:

\begin{theorem}\label{ThmAlmost}
The ring $(\Z;0,1,+,\cdot,=)$ is uniformly positive existentially interpretable in the class of $L_t$ structures $\Acal_R$ as $R$ ranges over all the integral domains of characteristic $p\ge 17$ endowed with a non-Archimedean absolute value.
\end{theorem}


\subsection{Step 3 and conclusion of the proof} Consider for each $p$ the class of $L_t$ structures $\Ucal_p$ that consists of all $\Acal_R$ with $\cha R=p$. The positive existential $L_t$-formula
$$
\kappa_p:\quad \underbrace{1+1+\cdots+1}_{p\, \mathrm{times}}=0
$$
is satisfied by each structure in $\Ucal_p$, and it is false in each $\Acal_R$ with $\cha R\ne p$. On the other hand, consider the class of $L_t$-structures $\Ucal_{\ge p}$ which is the union of the $\Ucal_q$ for all primes $q\ge p$. For $p>2$ given let $p_1,\ldots, p_r$ be the primes smaller than $p$. The positive existential $L_t$-formula
$$
\kappa_{\ge p}:\quad \exists z_1,\ldots,z_r, \bigwedge_{j=1}^r \underbrace{z_j+z_j+\cdots+z_j}_{p_j\, \mathrm{times}}=1
$$
is satisfied by each structure in $\Ucal_{\ge p}$, and it is false in each $\Acal_R$ with $0<\cha R< p$.

Let $\Ucal$ be the class of $L_t$-structures $\Acal_R$ where $R$ ranges over all integral domains $R$ of positive characteristic endowed with a non-Archimedean absolute value. Note that $\Ucal$ is the union of $\Ucal_2$, $\Ucal_3$, ..., $\Ucal_{13}$, $\Ucal_{\ge 17}$. By Theorem \ref{ThmCharp} we know that the ring $(\Z;0,1,+,\cdot,=)$ is uniformly positive existentially interpretable in each of $\Ucal_2$, $\Ucal_3$, ..., $\Ucal_{13}$, while Theorem \ref{ThmAlmost} shows that $(\Z;0,1,+,\cdot,=)$ is uniformly positive existentially interpretable in $\Ucal_{\ge 17}$. The formulas $\kappa_2$, ..., $\kappa_{13}$, $\kappa_{\ge 17}$ allow us to apply Fact 3.5 from \cite{PPV}, and we conclude that $(\Z;0,1,+,\cdot,=)$ is uniformly positive existentially interpretable in $\Ucal$. This concludes the proof of Theorem \ref{MainUnif}.
\section{Open problems}\label{SecOpen}

In this section we briefly present two open problems naturally related to the topics discussed in this work.

\subsection{The Archimedean case} After Theorem \ref{MainThm}, the only case of Hilbert's tenth problem (with coefficients in $\Z[t]$) for entire analytic functions in one variable  that remains unsolved is the Archimedean case. In particular, the following problem is open (with the notation introduced in Section \ref{SecIntro}).

\begin{problem} Is the positive existential theory of the ring of complex entire functions $\Hcal$ over the language $L_t$ decidable?
\end{problem}

For additional information on this problem, see \cite{PhZaSurvey} (specially sections 6 and 8) and the references therein.

\subsection{More general Pell equations} The Pell equation $X^2-(t^2-1)Y^2=1$ plays a central role in our arguments. In the case of entire functions in characteristic zero, analyzing this equation is enough for solving Hilbert's tenth problem (see \cite{LipPhe} or Theorem \ref{ThmChar0}) in the same way as in the work of Denef for polynomials in characteristic zero \cite{Denef0}. For the case of positive characteristic discussed in the present work, this equation is not enough for applying Denef's argument for polynomials in positive characteristic \cite{Denefp}, but we found a different approach without considering more general Pell equations. Nevertheless, the following problem is interesting on its own right:

\begin{problem}\label{ProblemPell}
Suppose that $\cha R\ne 0$. Let $x_n(t),y_n(t)\in R[t]$ be the polynomials defined by $x_n+\sqrt{t^2-1}y_n=(t+\sqrt{t^2-1})^n$. Let $f\in \Acal_R$ be non-constant. Is it true that the only solutions of $X^2-(f^2-1)Y^2=1$ in $\Acal_R$ are $(\pm x_n(f),y_n(f))$?
\end{problem}

A similar question can be asked when $\cha R=2$. This problem is solved (affirmatively) when $R$ has the trivial valuation, that is, when $\Acal_R=R[t]$ (see \cite{Denefp, PheidasZahidi}). In the Archimedean case $R=\C$ this problem has a negative answer already when $f=t$ (for instance, see Lemma C.6 in \cite{PhZaSurvey}). To the best of our knowledge, Problem \ref{ProblemPell} is open in all other cases, namely, when $R$ has a non-trivial non-Archimedean absolute value.

If Problem \ref{ProblemPell} has a positive answer for suitable $R$, then the arguments of \cite{PheidasZahidi} could be adapted to show undecidability of the positive existential theory of $\Acal_R$ over the `geometric' language $L_T=\{0,1,+,\cdot,=,T\}$, where $T(f)$ means `$f$ is non-constant'.

\section{Acknowledgments}

This project initiated during an informal seminar held by the authors at Queen's University in August 2013, in preparation for the AIM workshop \emph{Definability and Decidability Problems in Number Theory}, September 2013, Palo Alto. Most of the results were proved soon after the workshop, and the authors were greatly benefited from discussions with other participants. It is our pleasure to thank the American Institute of Mathematics and the NSF for their generous support, and the organizers of this workshop for the wonderful job that they did. 

We also thank Thanases Pheidas, Bjorn Poonen and Xavier Vidaux for their comments and suggestions on preliminary versions of this manuscript. 

Most of this work was done while both authors were graduate students at Queen's University. We thank the Department of Mathematics and Statistics at Queen's University for their support.

N. G.-F. was partially supported by a Becas Chile Scholarship and H. P. was partially supported by an Ontario Graduate Scholarship and by a Benjamin Peirce Fellowship.



\begin{thebibliography}{9}                                                                               


\bibitem{Wangetal} T. T. H. An, A. Levin, J. T.-Y. Wang, \emph{A p-adic Nevanlinna-Diophantine correspondence}. Acta Arith. 146 (2011), no. 4, 379-397. 

\bibitem{Ax} J. Ax, \emph{Solving Diophantine problems modulo every prime}. Ann. Math. 85(2), 161-183 (1967).

\bibitem{Berkovich} V. Berkovich, \emph{Spectral theory and analytic geometry over non-Archimedean fields.} Mathematical Surveys and Monographs, 33. American Mathematical Society, Providence, RI, 1990. x+169 pp. ISBN: 0-8218-1534-2

\bibitem{Denef0}  J. Denef, \emph{The Diophantine problem for polynomial rings and fields of rational functions}. Trans. Amer. Math. Soc. 242 (1978), 391-399. 

\bibitem{Denefp}  J. Denef, \emph{The Diophantine problem for polynomial rings of positive characteristic}. Logic Colloquium '78 (Mons, 1978), pp. 131-145, Stud. Logic Foundations Math., 97, North-Holland, Amsterdam-New York, 1979. 

\bibitem{HuYang}  P. C. Hu, C.-C. Yang, \emph{Value distribution theory related to number theory}. Birkh\"auser Verlag, Basel, 2006. xii+543 pp. ISBN: 978-3-7643-7568-3; 3-7643-7568-X


\bibitem{collected} L. Lipshitz; \emph{Quadratic forms, the five square problem, and diophantine equations, The collected works of J. Richard B\"uchi} (S. MacLane and Dirk Siefkes, eds.) Springer, 677-680, (1990).


\bibitem{LipPhe}  L. Lipshitz, T. Pheidas, \emph{An analogue of Hilbert's tenth problem for p-adic entire functions}. J. Symbolic Logic 60 (1995), no. 4, 1301-1309. 

\bibitem{Matijasevich} J. Matijasevic,  \emph{The Diophantineness of enumerable sets}. (Russian) Dokl. Akad. Nauk SSSR 191 1970 279-282. 

\bibitem{MazurBuchi} B. Mazur; \emph{Questions of decidability and undecidability in number theory}, J. of Symbolic Logic {\bf 59-2}, 353-371 (1994).

\bibitem{Survey} H. Pasten, T. Pheidas, X. Vidaux; \emph{A survey on B\"uchi's problem: new presentations and open problems.} Zap. Nauchn. Sem. S.-Peterburg. Otdel. Mat. Inst. Steklov. (POMI) 377 (2010), Issledovaniya po Teorii Chisel. 10, 111-140, 243; translation in J. Math. Sci. (N. Y.) {\bf 171} (2010), no. 6, 765-781 

\bibitem{PPV} H. Pasten, T. Pheidas, X. Vidaux, \emph{Uniform existential interpretation of arithmetic in rings of functions of positive characteristic}. Inventiones Mathematicae 196 (2014), no. 2, 453-484. 

\bibitem{PastenWang} H. Pasten, J. Wang, \emph{Extensions of B\"uchi's higher powers problem to positive characteristic}, to appear in IMRN (2014). 

\bibitem{PheidasHolo} T. Pheidas, \emph{The Diophantine theory of a ring of analytic functions}. J. Reine Angew. Math. 463 (1995), 153-167.

\bibitem{PheidasII}  T. Pheidas, \emph{An undecidability result for power series rings of positive characteristic. II}. Proc. Amer. Math. Soc. 100 (1987), no. 3, 526-530.

\bibitem{PheVid} T. Pheidas, X. Vidaux, \emph{The analogue of B\"uchi's problem for rational functions}. J. London Math. Soc. (2) 74 (2006), no. 3, 545-565. 


\bibitem{Corr} T. Pheidas, X. Vidaux, \emph{Corrigendum: The analogue of B\"uchi's problem for rational functions}. J. Lond. Math. Soc. (2) 82 (2010), no. 1, 273-278.

\bibitem{PheidasZahidi} T. Pheidas and K. Zahidi, \emph{Undecidable existential theories of polynomial rings and function fields}, Communications in Algebra, \textbf{27-10} 4993-5010 (1999).

\bibitem{PhZaSurvey} T. Pheidas, K. Zahidi, \emph{Undecidability of existential theories of rings and fields: a survey}. Hilbert's tenth problem: relations with arithmetic and algebraic geometry (Ghent, 1999), 49-105, Contemp. Math., 270, Amer. Math. Soc., Providence, RI, 2000.

\bibitem{ShlVid}  A. Shlapentokh, X. Vidaux, \emph{The analogue of B\"uchi's problem for function fields}. J. Algebra 330 (2011), 482-506. 

\bibitem{VojtaCIME} P. Vojta, \emph{Diophantine approximation and Nevanlinna theory}. Arithmetic geometry, 111-224, Lecture Notes in Math., 2009, Springer, Berlin, 2011.


\end{thebibliography}
\end{document}